\documentclass[10pt]{article}
\usepackage{graphicx} % Required for inserting images
\usepackage{amsmath}
\usepackage{amsthm}
\usepackage[a4paper, total={5in, 8in}]{geometry}
\usepackage{amssymb}
\usepackage{graphicx}
\usepackage{enumitem}
\newtheorem{theorem}{Theorem}[section]
\newtheorem{proposition}{Proposition}

\newtheorem{lemma}{Lemma}[section]

\theoremstyle{definition}
\newtheorem{definition}{Definition}
\newtheorem{remark}{Remark}

\numberwithin{equation}{section}

\begin{document}

\long\def\symbolfootnote[#1]#2{\begingroup%
\def\thefootnote{\fnsymbol{footnote}}\footnote[#1]{#2}\endgroup}
\title{\textbf{A class of symbols that induce bounded composition operators for Dirichlet-type spaces on the disc.}}
\author{Athanasios Beslikas}
\maketitle

\begin{abstract}
In this note we study the problem of determining the holomorphic self maps of the unit disc that induce a bounded composition operator on Dirichlet-type spaces. We find a class of symbols $\varphi$ that induce a bounded composition operator on the Dirichlet-type spaces, by applying results of the multidimensional theory of composition operators for the weighted Bergman spaces of the bi-disc.\\\\
\textbf{Mathematics Subject Classification:} primary: 30H99, secondary: 30H20.\\
\textbf{Keywords:} \textit{Dirichlet-type spaces, composition operators, weighted Bergman spaces.}
 \end{abstract}
\maketitle
 \section{Introduction} The problem of finding sufficient and necessary conditions for a composition operator to be bounded and compact has attracted a lot of attention the past few years. The resolution of the boundedness and compactness of composition operators on Hardy and Bergman spaces on the unit disc $\mathbb{D}$ is regarded as a classical result in holomorphic function spaces theory, and it was treated by Shapiro in \cite{Shapiro}. Since then, many other seminal works have provided us with characterizations of the holomorphic self maps $\varphi:\mathbb{D}\to \mathbb{D}$ that induce a bounded composition operator. For instance, the Dirichlet type spaces have been treated by J.Pau and P.Perez in \cite{Pau}. In this paper the authors gave necessary and sufficient conditions, for both continuity and compactness of the composition operator on the Dirichlet-type space. The \textit{Nevanlinna} and \textit{Generalized Nevanlinna counting function} was the main tool that helped significantly the authors of the abovementioned article to obtain these exceptional results. The interested reader can check the mentioned sufficient and necessary condition that the authors of \cite{Pau} provide in Theorem 3.1. of the mentioned paper. One observes that this characterization is obtained for specific exponents of the radial weight and not for all positive exponents.\\
 Compared to [7], in this note we will not make use of such tools and we will follow a completely different and somewhat novel approach, that combines results from multidimensional cases, namely the bi-disc case. The work of \L{}.Kosinski  in \cite{Kosinski} and F.Bayart \cite{Bayart2} dealt with the boundedness of composition operators in the weighted Bergman spaces of the bi-disc and provided us with characterizations of symbols that induce bounded composition operators. We will provide only a sufficient condition, which works for all positive exponents of the radial weight of the Dirichlet space, which is somewhat interesting to notice.\\
 In our work we present a way to exploit these multidimensional results, to obtain bounded composition operators for the unit disc case. In particular, we give a sufficient condition for some symbols $\varphi$ to induce a bounded composition operator for the Dirichlet spaces $\mathcal{D}_p(\mathbb{D}), p>0,$ by applying one of the main results presented in \cite{Kosinski}, namely the \textit{rank sufficiency Theorem} for the weighted Bergman space of the bi-disc. Our approach can be  described in a short manner as follows:
 Considering some holomorphic self map $\varphi$ of the unit disc, which is of class $\mathcal{C}^1$ on the boundary of the disc $\mathbb{T}$, we induce a composition operator $C_{\Psi}$ which is continuous on the weighted Bergman space $A^p_{\beta}(\mathbb{D}^2)$ of the bidisc, where $\Psi$ will be defined in terms of $\varphi.$ Then, by using the Lift operator that we introduce on Section 2 and a, recently obtained, double integral characterization of the Dirichlet-type space $\mathcal{D}_p(\mathbb{D})$ that can be found in \cite{Balooch}, we give a sufficient condition for the composition operator $C_{\varphi}$ to be bounded on $\mathcal{D}_p(\mathbb{D}).$
 \section{Notations and tools.}
 Throughout our note, we will denote by $\mathcal{O}(\mathbb{D},\mathbb{D})\cap \mathcal{C}^1(\overline{\mathbb{D}})$ the holomorphic self maps of the unit disc that are of class $\mathcal{C}^1$ on the boundary. In the same notation, by replacing $\mathbb{D}$ with $\mathbb{D}^2$ we will talk about the holomorphic self maps of the bi-disc that are of class $\mathcal{C}^1$ at the (topological) boundary of the bidisc, the bi-torus $\mathbb{T}^2$. Whenever we refer to a holomorphic function on the disc, we will simply write $f\in H(\mathbb{D}).$ Whenever the notation $a\asymp b$ appears, it means that there exist two positive constants $C_1,C_2$ such that $C_1a\leq b\leq C_2 a$ and whenever we encounter the notation $a\lesssim b,$ it means that there is a positive constant $C>0$ such that $a\leq Cb.$  We recall that for $p>0$ the Dirichlet type space of the unit disc, will be denoted by $\mathcal{D}_p(\mathbb{D}),$ and they consist of the holomorphic functions $f$ on the unit disc $\mathbb{D},$  such that
 $$\int_{\mathbb{D}}|f'(z)|^2(1-|z|^2)^pdA(z)<+\infty.$$
By $A_{\beta}^p(\mathbb{D}^2),\beta \ge -1,$ we denote the classical weighted Bergman space of the bi-disc $\mathbb{D}^2,$ consisted of the holomorphic functions on the bi-disc, such that:
$$\int_{\mathbb{D}^2}|f(z,w)|^pdA_{\beta}(z)dA_{\beta}(w)<+\infty.$$
The measure $dA(z)=\frac{1}{\pi}dxdy$ is the normalized Lebesgue measure on the unit disc, while $dA_{\beta}(z)=c_{\beta}(1-|z|^2)^{\beta}dA(z)$ for $\beta \ge -1.$ The following recent result, that can be found on \cite{Balooch}, will be of critical importance in our note.
\begin{theorem} \textit{Let $\sigma, \tau \ge -1 $ and $\beta\in \mathbb{R},$ such that $\frac{\max(\sigma,\tau)}{2}-1<\beta\leq \frac{\sigma+\tau}{2}.$ Let $f\in H(\mathbb{D}).$ Then:}
$$\int_{\mathbb{D}}\int_{\mathbb{D}}\frac{|f(z)-f(w)|^2}{|1-\overline{w}z|^{2(\beta+2)}}dA_{\sigma}(z)dA_{\tau}(w)\asymp ||f||^2_{\mathcal{D}_{\sigma+\tau-2\beta}(\mathbb{D})}.$$
\end{theorem}
We define now a lift-type operator, similarly to \cite{Zhu}, this time with an exponent of two positive parameters $p,\gamma>0$. 
\begin{definition} \textit{Let $p,\gamma>0$ and $f\in H(\mathbb{D}).$ The lift-type operator is defined as follows:  }
$$L^{p,\gamma}:H(\mathbb{D})\to H(\mathbb{D}^2),$$
$$L^{p,\gamma}(f):=\frac{f(z)-f(w)}{(1-\overline{w}z)^{p/ \gamma}},z,w \in \mathbb{D}$$
\end{definition}
The first observation that one makes immediately is that the fraction that defines this operator appears in the above mentioned characterization of the Dirichlet-type spaces. The following proposition follows immediately:
\begin{proposition} \textit{Let $f\in H(\mathbb{D})$ and $\sigma=\tau > -1$ and $\beta$ as in Theorem 2.1. The operator $L^{2,2(\beta+2)}(f)$ maps $\mathcal{D}_{2\sigma-2\beta}(\mathbb{D})$ into $A^2_{\sigma}(\mathbb{D}^2).$}
\end{proposition}
\begin{proof}
The proof follows by a simple use of the asymptotic equality of Theorem 2.1 but we will write it down for convenience.
\begin{align}
||L^{2,2(\beta+2)}(f)||^2_{A^2_{\sigma}(\mathbb{D}^2)}&=\int_{\mathbb{D}}\int_{\mathbb{D}}\left|L^{2,2(\beta+2)}(f)(z,w)\right|^2dA_{\sigma}(z)dA_{\sigma}(w)&&\\
&=\int_{\mathbb{D}}\int_{\mathbb{D}}\left|\frac{f(z)-f(w)}{(1-\overline{w}z)^{\frac{2(\beta+2)}{2}}}\right|^2dA_{\sigma}(z)dA_{\sigma}(w)&&\\
&\lesssim ||f||^2_{\mathcal{D}_{2\sigma-2\beta}(\mathbb{D})}.
\end{align}
\end{proof}
The next theorem that we will state is a characterization of the symbols $\Psi \in \mathcal{O}(\mathbb{D}^2,\mathbb{D}^2)\cap \mathcal{C}^1(\overline{\mathbb{D}^2})$ that induce a bounded composition operator on the weighted Bergman space $A_{\beta}^2(\mathbb{D}^2),\beta>-1$ and is due to \L{}.Kosinski, found in \cite{Kosinski},(section 1, page 3)
\begin{theorem} \textit{Let $\Psi \in \mathcal{O}(\mathbb{D}^2,\mathbb{D}^2)\cap \mathcal{C}^1(\overline{\mathbb{D}^2}). $ Then, $C_{\Psi}:A_{\beta}^2(\mathbb{D}^2)\to A_{\beta}^2(\mathbb{D}^2)$ is bounded, if and only if the derivative $d_{\zeta}\Psi$ is invertible for all $\zeta\in\mathbb{T}^2$ such that $\Psi(\zeta)\in\mathbb{T}^2.$}
\end{theorem}
Using this result, we can deduce the following lemma:
\begin{lemma}
\textit{Let $\varphi \in \mathcal{O}(\mathbb{D},\mathbb{D})\cap \mathcal{C}^1(\overline{\mathbb{D}}).$ Set $\Phi(z_1,z_2)=(\varphi(z_1),\varphi(z_2)),z_1,z_2 \in \mathbb{D}.$ Then $C_{\Phi}:A_{\beta}^2(\mathbb{D}^2)\to A_{\beta}^2(\mathbb{D}^2)$ is bounded.}
\end{lemma}
\begin{proof} It is quite obvious that $\Phi \in \mathcal{O}(\mathbb{D}^2,\mathbb{D}^2)\cap \mathcal{C}^1(\overline{\mathbb{D}^2}). $ We only have to show that for every $\zeta\in \mathbb{T}^2$ such that $\Phi(\zeta)\in\mathbb{T}^2$, the derivative $d_{\zeta}\Phi$ is invertible.   Let $\zeta=(\zeta_1,\zeta_2)$ such that  $$\Phi(\zeta_1,\zeta_2)=(\varphi(\zeta_1),\varphi(\zeta_2))\in \mathbb{T}^2.$$ Of course, this only occurs for the points $\zeta_1,\zeta_2\in\mathbb{T}$ such that $\varphi(\zeta_1),\phi(\zeta_2)\in\mathbb{T},\zeta_1\neq\zeta_2.$  We calculate the Jacobian of the derivative $d_{\zeta}\Phi$ for such points $\zeta\in\mathbb{T}^2$. It is an immediate observation that $\frac{\partial \varphi(z_1)}{\partial z_2}=\frac{\partial \varphi(z_2)}{\partial z_1}=0.$ So the Jacobian is
$$|d_\zeta\Phi|=\varphi'(\zeta_1)\varphi'(\zeta_2)\neq 0.$$
By the fact that $\varphi$ is of class $\mathcal{C}^1$, hence continuous on the boundary, one can calculate the value of $\varphi'$ on the boundary points $\zeta_1,\zeta_2$. By the Julia-Caratheodory Theorem (see \cite{Caratheodory}) it is asserted that the value of the derivative of $\varphi$ is non-zero. Hence, by Theorem 2.2., $C_{\Phi}$ defines a bounded composition operator on the weighted Bergman space $A_{\beta}^2(\mathbb{D}^2),$ and the proof is complete.
\end{proof}
\section{Main result and proof} Here we state our main result and the proof of it. In what follows, we denote by 
$$k^{\varphi}(z,w)=\frac{1-\varphi(z)\overline{\varphi(w)}}{1-z\overline{w}},z,w\in \mathbb{D},$$
for a function $\varphi \in \mathcal{O}(\mathbb{D},\mathbb{D})\cap \mathcal{C}^1(\overline{\mathbb{D}}),$
and by $||\cdot||_{\infty}$ the supremum norm over the bi-disc $\mathbb{D}^2,$ that is $||k||_{\infty}=\sup_{(z,w)\in\mathbb{D}^2}|k(z,w)|$  . 
\begin{theorem} \textit{Let $\varphi \in \mathcal{O}(\mathbb{D},\mathbb{D})\cap \mathcal{C}^1(\overline{\mathbb{D}})$ such that}
$$||k^{\varphi}(z,w)||_{\infty}<+\infty.$$ 
\textit{Then $C_{\varphi}:\mathcal{D}_p(\mathbb{D})\to \mathcal{D}_p(\mathbb{D})$ is bounded, for $p=2\sigma-2\beta,$ $\sigma>0$ satisfying $\frac{\sigma}{2}-1<\beta<\sigma$.}
\end{theorem}
\begin{proof} Let $\sigma>0$ and $\beta\in \mathbb{R}$ satisfying the conditions that we gave in our statement. For convenience, set $q=2(\beta+2).$
\begin{align} 
||C_{\varphi}(f)||^2_{\mathcal{D}_{2\sigma-2\beta}(\mathbb{D})}&\asymp \int_{\mathbb{D}}\int_{\mathbb{D}}\frac{|f\circ\varphi(z)-f\circ \varphi(w)|^2}{|1-\overline{w}z|^{q}}dA_{\sigma}(z)dA_{\sigma}(w)&&\\
&=\int_{\mathbb{D}}\int_{\mathbb{D}}\frac{|f\circ\varphi(z)-f\circ \varphi(w)|^2}{|1-\overline{\varphi(w)}\varphi(z)|^{q}}\frac{|1-\overline{\varphi(w)}\varphi(z)|^{q}}{|1-\overline{w}z|^{q}}dA_{\sigma}(z)dA_{\sigma}(w)&&\\
&\leq \sup_{(z,w)\in \mathbb{D}^2}\left|k^{\varphi}(z,w)\right|^{q}\int_{\mathbb{D}}\int_{\mathbb{D}}\frac{|f\circ\varphi(z)-f\circ \varphi(w)|^2}{|1-\overline{\varphi(w)}\varphi(z)|^{q}}dA_{\sigma}(z)dA_{\sigma}(w).
\end{align}
 At this point, we observe that the integral on the right hand side can be expressed as the norm of a composition operator $C_{\Phi}$ on the weighted Bergman space $A_{\sigma}^2(\mathbb{D}^2),$ where $\Phi=\Phi(z_1,z_2)=(\varphi(z_1),\varphi(z_2)),z_1,z_2\in\mathbb{D}.$ To be precise:
$$\int_{\mathbb{D}}\int_{\mathbb{D}}\frac{|f\circ\varphi(z)-f\circ \varphi(w)|^2}{|1-\overline{\varphi(w)}\varphi(z)|^{2(\beta+2)}}dA_{\sigma}(z)dA_{\sigma}(w)=||C_{\Phi}(L^{2,2(\beta+2)}(f))||^2_{A^2_{\sigma}(\mathbb{D}^2)}.$$
By Lemma 2.1. we have that $C_\Phi:A^2_{\sigma}(\mathbb{D}^2)\to A^2_{\sigma}(\mathbb{D}^2)$ is bounded. Hence, by utilizing Proposition 1, we receive:
$$||C_{\Phi}(L^{2,2(\beta+2)}(f))||^2_{A^2_{\sigma}(\mathbb{D}^2)}\lesssim||L^{2,2(\beta+2)}(f)||^2_{A^2_{\sigma}(\mathbb{D}^2)}\lesssim ||f||^2_{\mathcal{D}_{2\sigma-2\beta}(\mathbb{D})}.$$
Summarizing, we obtained:
$$||C_{\varphi}(f)||^2_{\mathcal{D}_{2\sigma-2\beta}(\mathbb{D})}\lesssim ||k^{\varphi}(z,w)||_{\infty}^{2(\beta+2)}||f||^2_{\mathcal{D}_{2\sigma-2\beta}(\mathbb{D})},$$
which, by our assumption that the supremum is bounded, provides us with the desired result.
\end{proof}
\section{Concluding remark}
The significance of the main result of this note lies mainly in the fact that we managed to apply results from case of the bi-disc to the case of the unit disc. There exists another interesting aspect of the result, and that is the appearence of the \textit{De Branges-Rovnyak kernel} that is induced by the symbol $\varphi.$ Assuming that $\varphi \in \mathcal{O}(\mathbb{D},\mathbb{D})\cap \mathcal{C}^1(\overline{\mathbb{D}}),$ then it is known that $\varphi \in H^{\infty}(\mathbb{D})$ and also $||\varphi||_{H^{\infty}(\mathbb{D})}\leq 1$ which means that the function 
$$k^{\varphi}(z,w)=\frac{1-\varphi(z)\overline{\varphi(w)}}{1-z\overline{w}},z,w\in \mathbb{D}$$ is the \textit{De Branges-Rovnyak reproducing kernel} associated with $\varphi.$ We can reformulate the main result of the note in the following manner: 
\begin{remark}
\textit{Let $\varphi \in \mathcal{O}(\mathbb{D},\mathbb{D})\cap \mathcal{C}^1(\overline{\mathbb{D}})$ and assume that $|k^{\varphi}|$, ($k^{\varphi}$ is the \textit{De Branges-Rovnyak kernel}  associated to $\varphi$) is a bounded function on the bi-disc. Then, the composition operator $C_{\varphi}:\mathcal{D}_p(\mathbb{D})\to \mathcal{D}_p(\mathbb{D}), p>0 $ is bounded.}
\end{remark}
For more details about the De Branges-Rovnyak spaces and kernel, the reader can check \cite{DeBranges}. Also, in the work of M.Jury (see \cite{Jury}), one can study the connections of the De Branges-Rovnyak kernel associated to a holomorphic self map of the unit disc and the corresponding composition operators in the Hardy and Bergman spaces of the unit disc, but also of the unit ball of $\mathbb{C}^n.$

\section{Financial support-Acknowledgements.}
\textbf{Financial Support:}
The author was supported by the NCN grant SONATA BIS no.2017/26/E/ST1/00723.\\\\
\textbf{Acknowledgements}
The author would like to thank the following people: Prof.\L{}ukasz Kosinski for his valuable help, Dr. D.Vavitsas, Dr. A.Kouroupis and A.Tsiri for their friendship, and his family for the support. Also, the author would like to thank the anonymous referee for the useful comments and the careful reading of the paper.

Athanasios Beslikas,\\
Doctoral School of Exact and Natural Studies\\
Institute of Mathematics,\\
Faculty of Mathematics and Computer Science,\\
Jagiellonian University\\ 
\L{}ojasiewicza 6\\
PL30348, Cracow, Poland\\
athanasios.beslikas@doctoral.uj.edu.pl

\end{document}